\documentclass
{amsart}

\usepackage{mathtools,amsthm,amssymb,amsfonts, epsfig}
\usepackage{setspace, stmaryrd, verbatim, euro, enumerate}
\usepackage{bbm, tensor,comment,textcomp, todonotes}

\usepackage{hyperref}

\usepackage{color}

 \newtheorem{theorem}{Theorem}
\newtheorem{lemma}[theorem]{Lemma}

\theoremstyle{definition}

\theoremstyle{remark}

\newtheorem{remark}[theorem]{Remark}

\newcommand\cF{\mathcal{F}}
\newcommand\cG{\mathcal{G}}
\newcommand\cH{\mathcal{H}}

\newcommand\cP{\mathcal{P}}

\newcommand\cS{\mathcal{S}}

\newcommand\bE{\mathbb{E}}

\newcommand\bG{\mathbb{G}}

\newcommand\bN{\mathbb{N}}

\newcommand\bR{\mathbb{R}}

\DeclareMathOperator{\co}{co}
\DeclareMathOperator{\fcc}{fcc}

\begin{document}

\title{Differentiation of measures on a non-separable space, and the  Radon-Nikodym theorem  \medskip\\
%\texttt{DRAFT, Strictly confidential, please do not distribute}
}
%\subtitle{Part I}
\author{Oleksii Mostovyi}
  \author{Pietro Siorpaes}
 
% \date{\today{}}
 \address{Oleskii Mostovyi, 
University of Connecticut \\
Department of Mathematics\\
Storrs, CT 06269, United States
}
\email{oleksii.mostovyi@uconn.edu}
\address{Pietro Siorpaes: Department of Mathematics, Imperial College London, London SW7 2AZ, United Kingdom}
\email{p.siorpaes@imperial.ac.uk}

\begin{abstract} 
Given positive measures $\nu,\mu$ on an arbitrary measurable space $(\Omega, \cF)$, we \emph{construct a sequence} of finite partitions $(\pi_n)_n$ of $(\Omega, \cF)$ s.t. 
$$ \sum_{A\in \pi_n: \mu(A)>0} 1_{A} \frac{\nu(A)}{\mu(A)} \longrightarrow  \frac{d\nu^a}{d\mu} \quad \mu \text{ a.e. as } n\to \infty .
$$
As an application, we modify the probabilistic proof of the Radon-Nikodym Theorem so that it uses convergence along a properly chosen sequence (instead of along a net), and so that it does not rely on  the martingale convergence theorem (nor any probability theory),  obtaining a  completely elementary proof. 
\end{abstract} 
%\thanks{}
\maketitle
%\section{}

The following theorem, due to Lebesgue\footnote{Although this is commonly referred to as the Radon-Nikodym theorem, the first version of the existence of the density of a measure on $\bR^n$ absolutely continuous with respect to the Lebesgue measure, is due to  Lebesgue; Radon extended this result  to Radon measures, and Nikodym to general measures (see \cite[footnote 18, p. 155]{DiDe02}). Moreover, the existence of the decomposition $\nu=\nu^a+\nu^s$ is also due to Lebesgue.}, Radon and Nikodym, has been called  `probably the most important theorem in measure theory' in the classic   book \cite{Ru06}.
\begin{theorem}
\label{RN}
Given finite positive (sigma-additive) measures $\nu,\mu$ on a measurable space $(\Omega, \cF)$, there exists unique  positive measures $\nu^a, \nu^s$ s.t. $\nu=\nu^a+\nu^s$, $\nu^a \ll \mu$ and $\nu^s\perp \mu$, and there exists unique  $f\in L^1(\mu):=L^1(\Omega, \cF, \mu)$ s.t. $\nu^a=f \cdot \mu$.
\end{theorem} 
To clarify, we denote with $f \cdot \mu$ the measure defined by 
$$(f  \cdot \mu)(A):=\int_A f d\mu \text{ for all } A\in \cF .$$
 Of course Theorem \ref{RN} admits  variants for the cases of real, complex, and sigma-finite  measures, which readily follow from the statement above. There are also more exotic extensions, for example \cite{ha13} goes somewhat beyond $\sigma$-finiteness, and \cite[Chapter 5]{EdSu92} considers Banach-valued $\nu$.
 
A way to \emph{construct} the function $f=\frac{d\nu^a}{d\mu}$ if $\Omega=\bR^N$ is using the following classical theorem of differentiation of measures (see \cite[Chapter 1, Section 6]{EvGa18}, which calls it `the fundamental theorem of calculus for Radon measures in $\bR^n$').
%or \cite[Chapter 4, Sections 8-11]{DiDe02}
\begin{theorem}
\label{thm: diff meas}
Given $\mu,\nu$  positive Borel measures on $\bR^n$, finite on compacts, let\footnote{Here we use the convention that  $h_{\epsilon}(x):=\infty$ for all $x$ for which  $\mu(B_\epsilon(x))=0$.}  
\begin{align*}
h_{\epsilon}(x):=\frac{\nu(B_\epsilon(x))}{\mu(B_\epsilon(x))}  , \quad \text{ for } B_\epsilon(x):=\{ y\in \bR^n: ||y-x||\leq \epsilon\} .
\end{align*} 
Then   $ h_{\epsilon}\to \frac{d\nu^a}{d\mu}$ $\mu$ a.e. as $\epsilon\downarrow 0$.
\end{theorem}
Analogously, if $(\Omega,\cF)$ is \emph{separable}, i.e. there exists a sequence of  sets $(B_j)_{j\in \bN}$ s.t. $\cF=\sigma((B_j)_{j\in \bN})$, and $\pi_n$ is  the\footnote{The elements of $\pi_n$ are the atoms of $\sigma((B_j)_{j=0}^n)$, and are the sets of the form $\cap_{j=0}^n C_j$ where $C_j\in \{B_j, \Omega \setminus B_j\}$.} 
partition  of $\Omega$ s.t. $\sigma(\pi_n)=\sigma((B_j)_{j=0}^n)$, then 
\begin{align}
\label{thm: mart diff}
f_{\pi_n}:=f_{\pi_n}(\mu):=\sum_{A\in \pi_n: \mu(A)>0} 1_{A} \frac{\nu(A)}{\mu(A)} \longrightarrow  \frac{d\nu^a}{d\mu} \quad \mu \text{ a.e. as } n\to \infty ,
\end{align} 
and if $\nu \ll \mu$ the convergence is also in $L^1(\mu)$ (even in $L^p$  if $\frac{d\nu}{d\mu}\in L^p$ and $p\in [1,\infty)$). 
This interesting fact, which seems unfortunately little known to  non-probabilists, closely resembles  Theorem \ref{thm: diff meas}: the main difference is that to build $f_{\pi_n}$ in \eqref{thm: mart diff} one evaluates $\frac{\nu(A_x)}{\mu(A_x)}$ at the set $A_x\in \pi_n$ which contains $x$, where the family $\pi_n$ is  \emph{fixed}, i.e. $\pi_n$ does not depend on $x$. 
The martingale-based method used to prove \eqref{thm: mart diff} can also be used to investigate what families of sets one can use in  Theorem \ref{thm: diff meas} instead of 
$(B_\epsilon(x))_{\epsilon>0, x\in \bR^n}$; for an exhaustive study of the topic of derivation and its relation to martingales one can consult  \cite{HaPa12}, and for a shorter and readable account of the most important results see  \cite[Chapter 7]{EdSu92}.

For arbitrary $(\Omega,\cF)$, it is also known that, if $\cP$ is the family of all finite partitions\footnote{I.e. any $\pi\in \cP$ if of the form $(A_j)_{j=0}^n$ for some $n\in \bN$, where the $A_j$'s are $\cF$-measurable, disjoint and their union is $\Omega$.} of $(\Omega,\cF)$, ordered by refinement, then the  net $(f_{\pi}(\mu))_{\pi\in \cP}$  converges to $\frac{d\nu}{d\mu}$ in $L^1(\mu)$ if $\nu \ll \mu$ (but not otherwise).

Our main contribution is then to generalize \eqref{thm: mart diff} to non-separable $(\Omega,\cF)$, 
by identifying a \emph{sequence} of partitions  $(\pi_n)_{n\in \bN}$ such that  $f_{\pi_n}(\mu)\to \frac{d\nu^a}{d\mu}$ $\mu$ a.e., as follows: 
\begin{theorem}
\label{thm: conv seq}
If  $\pi_n\in \cP, n\in \bN$ is increasing and chosen such that $(f_{\pi_n}(\gamma))_n$, defined via \eqref{thm: mart diff} using $\gamma:=\mu+\nu$, asymptotically maximizes\footnote{I.e. satisfies
$$ \sup_n \int_{\Omega} e^{-f_{\pi_n}(\gamma)} d\gamma=
 \sup_{\pi\in \cP} \int_{\Omega} e^{-f_{\pi}(\gamma)} d\gamma . $$} 
the function $\cP \ni \pi \mapsto \int_{\Omega} e^{-f_{\pi}(\gamma)} d\gamma$, then $f_{\pi_n}(\mu)\to  \frac{d\nu^a}{d\mu}$ $\mu$ a.e..
\end{theorem} 
Of course, if $\nu \ll  \mu$ then $(f_{\pi}(\mu))_{\pi\in \cP}$  is uniformly $\mu$-integrable and so the convergence $f_{\pi_n}(\mu)\to  \frac{d\nu^a}{d\mu}$ is also in $L^1(\mu)$.
It is not obvious how to justify the intuition behind the choice of $\pi_n$ in Theorem \ref{thm: conv seq}. This we do in Section \ref{se:analogies}, by introducing  an order which is closely linked to the martingale property; this point of view seems new, and proves fruitful.
On the other hand, once made the right guess for $\pi_n$, the
 proofs of Theorems \ref{RN} and \ref{thm: conv seq} are quite easy.

\begin{remark}
\label{rem: convergence for general nu}
In  \cite[Chapter 11, Section 17]{do12} it is shown that if $\pi_n \in \cP, n\in \bN$ is increasing (i.e. refining)  then 
\begin{align}
\label{}
f_{\pi_n} \to  \bE^\mu\left[\frac{d\nu^a}{d\mu}\Big|\tilde{\cF}\right] \mu \text{ a.e. as } n\to \infty ,
\end{align} 
 where $\tilde{\cF}:=\sigma(\cup_n \cF_{\pi_n})$ and $\bE^\mu[\cdot |\tilde{\cF}]$ denotes the conditional $\mu$-expectation w.r.t. $\tilde{\cF}$.  So,  while $\cF$ will in general\footnote{In particular this happens whenever $\cF$ is not separable.} be strictly bigger than  $\tilde{\cF}:=\sigma(\cup_n \cF_{\pi_n})$, with the $\pi_n$ as in Theorem \ref{thm: conv seq} one gets (in hindsight) that $\frac{d\nu^a}{d\mu}$ is $\tilde{\cF}$-measurable; this is the ultimate reason why it is enough to take the limits along the sequence we chose, instead of using the whole net. Notice that it is 
obvious\footnote{Any $\cF$-measurable $g:\Omega\to \bR$ is also $\tilde{\cF}$-measurable for some separable sigma algebra $\tilde{\cF} \subseteq  \cF$, because the Borel sets of $\bR$ form a separable  sigma algebra, and so also $\sigma(g)$ is separable.}
that $\frac{d\nu^a}{d\mu}$ is $\sigma(\cup_n \cF_{\pi_n})$-measurable for some choice of $\pi_n$, and so it was already known that $f_{\pi_n}(\mu) \to \frac{d\nu^a}{d\mu}$ $\mu$ a.e. for some choice of $(\pi_n)_n$. An added value that Theorem \ref{thm: conv seq} brings, is that it specifies $(\pi_n)_n$  explicitly via \eqref{ChoiceofSub}, asking not that the unknown quantity  $\frac{d\nu^a}{d\mu}$ be $\sigma(\cup_n \cF_{\pi_n})$-measurable, but rather that $(\pi_n)_n$ asymptotically maximizes the function $\pi \mapsto \int_{\Omega} e^{-f_\pi} d\gamma$, whose values can be calculated from the known quantities $\mu,\nu$. Another added value of Theorem \ref{thm: conv seq} is that it allows to give an elementary proof of Theorem \ref{RN}, as follows.
\end{remark}

As an application of Theorem \ref{thm: conv seq}, we modify the probabilistic proof  of Theorem \ref{RN} so that it uses convergence along a properly chosen  sequence instead of along a net. To appreciate the improvement, consider that proving Theorem \ref{RN} in the non-separable case is inconvenient enough that  \cite{JaPr03Pr} simply skips the proof of this more technical case, and \cite{Wi91} breaks the proof into the separable and general case, and says \emph{`Proving Part II of the theorem is a piece of 'abstract nonsense' [...] You might well want to take Part II for granted and skip the remainder of this section'}.

 In fact, we chose to prove Theorem \ref{RN}  \emph{without}\footnote{In this regard we mention that in  \cite{EdSu76Am} and \cite[Theorem 1.3.2]{EdSu92} it is shown how Theorem \ref{RN} can be proved without  the full power of the martingale convergence theorem, but rather relying on a related, less well known but more elementary convergence theorem for amarts (a.k.a.  asymptotic martingales). 
.}
  even using  the  martingale convergence theorem,  by noticing that instead of the $\mu$ a.e. convergence of the uniformly integrable sequence $(f_{\pi_n})_n$ ensured by Theorem \ref{thm: conv seq}, it is enough to prove the $L^1(\mu)$ convergence of a forward convex combination of $(f_{\pi_n})_n$ (which is easy). Using this weaker version of Theorem \ref{thm: conv seq} proved in the course of the proof of Theorem \ref{RN}, and the martingale convergence theorem, we then prove Theorem \ref{thm: conv seq}.
 While this approach lengthens the  proofs of Theorems \ref{RN} and \ref{thm: conv seq} a little, we obtain a proof of Theorem \ref{RN} which is completely elementary, just like its two most  popular proofs  (discussed in Section \ref{se:several proofs}), and is  purely analytic (no knowledge of probability theory is required).

\medskip

 Unsurprisingly, there exists several proofs of Theorem \ref{RN}. In Section \ref{se:several proofs} we go over the ones we are aware of. In Section \ref{se:several proofs} (and, at a different level in Section \ref{se:analogies}) we also  discuss some interesting analogies between our proof and the most popular proof of  Theorem \ref{RN}, which we hope leads to a deeper understanding of both proofs. Section \ref{se:proof} contains the proofs of Theorems \ref{RN} and \ref{thm: conv seq}, tersely  written, while Section \ref{se:analogies} discusses and clarifies the choice of the properly chosen  sequence $(\pi_n)_n$ in our proof. 
 
\section{Analogies between proofs of Theorem \ref{RN}.}
\label{se:several proofs}

The most popular proof  of Theorem \ref{RN} identifies  $f=\frac{d\nu^a}{d\mu}$  as the $\mu$-essential supremum (denoted with $\mu\text{-}esup$) of 
\begin{align}
\label{eq:fesssup}
\textstyle
L(\mu,\nu):=\{g\in L^1_+(\mu):  \int_A g d \mu  \leq \nu(A) \text{ for all  } A\in \cF\} ,
\end{align} 
i.e. the supremum of $L(\mu,\nu)$ seen as a subset of the space of equivalence classes $L^1(\mu)$.
Variants of this proof appear in the classic texts  \cite{fol99},  \cite{ha13}, \cite{Ro88Re}, \cite{DiDe02}, \cite{Bogachev}, \cite{Bil}, \cite{Yeh}, which however do not explicitly clarify the connection with the essential supremum; and \cite{ChTe03}, which does. A more complicated but somewhat related proof can be found  in \cite{WhZy77}.

The other popular proof of Theorem \ref{RN}  is due to Von Neumann and is based on the Riesz-Frechet Representation Theorem for the dual of the Hilbert space $L^2$; we refer to  \cite{Ru06}, \cite{Str98Co}, \cite{Kh07Pr}, \cite{Pedersen} for  variants of this proof. 

For proofs based on the theory of Riesz spaces, one can consult 
\cite{Za12In} and \cite{Lu79So}, which rely on Freudenthal's spectral theorem; or  \cite{Sc74Ba}, which relies on a characterization of order continuous forms on $L^\infty(\mu)$. 

In  \cite[Chapter 5, Items 56-57]{DeMeB} and \cite[Chapter 14, Section 13]{Wi91}, one can find the classic probabilistic proof of Theorem \ref{RN},  based on the martingale convergence theorem. A benefit of this proof is that it `gives easily the following extremely useful theorem, due to Doob' (see \cite[Chapter 5, Item 58]{DeMeB}, from which we took the quote): if $\mu_v,\nu_v$ are measures on $(\Omega,\cF)$ which depend measurably on some parameter $v$, \emph{and $\cF$ is separable}, then  one can choose a version $f(v,\cdot)$ of $d\nu_v^a/d\mu_v$ s.t. $f$ is jointly-measurable.

\medskip

Now let us explain (without assuming any knowledge of probability theory) the ideas behind   the most popular proof and the probabilistic proof, as to  highlight the close relationship between them (we highlight deeper analogies  in Section \ref{se:analogies}). When building the Lebesgue theory of integration, one approximates a positive function $f$ with the increasing sequence of step functions $0\leq f_n\leq f$  defined by 
$$f_n(x):=k/2^n \text{ for } x  \text{ s.t. } k/2^n\leq f(x) <(k+1)/2^n .$$ 
Correspondingly, one can approximate the measure $f\cdot \mu$ from below using $0\leq f_n\cdot \mu\leq f\cdot \mu$, where by definition  two real measures $\alpha,\beta$ satisfy $\alpha\leq \beta$ when $\alpha(A)\leq \beta(A)$ for all $A\in \cF$. Notice that  $f_n(x)\in D_n:=\{k/2^n: k\in \bN\}$ for all $x$. Just like $f_n$ is the supremum of all $D_n$-valued functions below $f$,  $f_n$ can be defined (up to $\mu$-null sets) using $f\cdot \mu$ instead of $f$, as the $\mu$-essential supremum of 
\begin{align}
L_n(\mu,\nu):=\{g\in L^1_+(\mu): g \in D_n \quad \mu \text{ a.e. },  \int_A g d \mu  \leq \int_A f d\mu \text{ for all  } A\in \cF\} , 
\end{align}
since two functions $g,h$ satisfy $g\leq h$ $\mu$ a.e. iff $g\cdot \mu \leq h\cdot \mu$.
Thus, if a measure $\nu$ satisfies $\nu=f\cdot \mu$, necessarily the functions 
\begin{align}
\label{eq:f_n esssup}
f_n:=\mu\text{-}esup \{g\in L^1_+(\mu): g \in D_n \quad \mu \text{ a.e. },  g \cdot \mu  \leq \nu \} , 
\end{align}
satisfy $0\leq f_n \uparrow f$ $\mu$ a.e.; so, if one does not know in advance whether $\nu$ is of the form $f \cdot \mu$, one simply has to check whether the limit $f$ of the increasing sequence of functions $f_n$ defined  by \eqref{eq:f_n esssup} satisfies $\nu=f\cdot \mu$. Clearly such $f$ equals $\mu\text{-}esup L(\mu,\nu)$, and the core of the most popular proof of Theorem \ref{RN} is indeed to show that such $f$ satisfies $\nu=f\cdot \mu$ if 
$\nu \ll \mu$. Thus, since  $g\cdot \mu \leq  \nu$ 
	implies\footnote{Let $B\in \cF$ be s.t. $\nu^s(B)=0=\mu(\Omega \setminus B)$, and $A\in \cF$. Then $\mu(A\setminus B)=0$ implies  $\nu^a(A\setminus B)=0$ and so
$$ \int_A g d \mu  = \int_{A \cap B} g d \mu  \leq \nu(A \cap B)=\nu^a(A\cap B)=\nu^a(A)  .$$
} 
	$g\cdot \mu \leq  \nu^a$, we have $L(\mu,\nu)=L(\mu,\nu^a)$, and so the identity $\frac{d\nu^a}{d\mu}=\mu\text{-}esup L(\mu,\nu)$ holds for any $\mu,\nu$.

\medskip

 An alternative approach to build $f$ s.t. $\nu= f \cdot \mu$ it to use the kind of approximation that one uses when building the Riemann integration theory. In this case we consider the family $\cP$  of all finite  partitions of $(\Omega,\cF)$. 
We can use  $\pi_:=(A_j)_j\in \cP$ to approximate $f$ by its local average 
\begin{align}
\label{eq: fpi def using f}
f_\pi=\sum_{j: \mu(A_j)>0} 1_{A_j} \frac{\int_{A_j} f d \mu}{\mu(A_j)} ;
\end{align} 
notice that $f_\pi$, like the $f_n$ defined in \eqref{eq:f_n esssup},  takes only finitely many values and can be defined using $\nu=f\cdot \mu$ instead of $f$ by setting 
\begin{align}
\label{eq: fpi def using nu}
 f_\pi:= f_\pi(\mu):=\sum_{j: \mu(A_j)>0} 1_{A_j} \frac{\nu(A_j)}{\mu(A_j)} .
\end{align} 
 Notice that we restrict  the measures  $\mu,\nu$  to the sigma algebra $\sigma(\pi)$ and then consider the Lebesgue decomposition of $\nu_{|\sigma(\pi)}$ into $(\nu_{|\sigma(\pi)})^a+(\nu_{|\sigma(\pi)})^s$, we find that  
\begin{align}
\label{eq:RN density}
\text{$f_{\pi}(\mu)= \frac{d(\nu_{|\sigma(\pi)})^a}{d\mu_{|\sigma(\pi)}}$;}
\end{align} 
in particular  if $\nu \ll \mu$ then 
 $f_\pi=\frac{d\nu_{|\sigma(\pi)}}{d\mu_{|\sigma(\pi)}}$,  since in this case  $\nu_{|\sigma(\pi)} \ll \mu_{|\sigma(\pi)}$. We'll need the following remark, which does not hold without the assumption $\nu \ll \mu$.

\begin{remark}
\label{rem:UI}
$(f_\pi)_{\pi\in \cP}$ is uniformly integrable if $\nu \ll \mu$.
\end{remark} 
\begin{proof} We will use  $f_\pi=\frac{d\nu_{|\sigma(\pi)}}{d\mu_{|\sigma(\pi)}}$ twice. Since $\nu \ll \mu$ and 
 $$ \mu(f_\pi\geq k)\leq \int_{\{ f_\pi\geq k\}} \frac{f_\pi}{k} d\mu \leq \int_{\Omega} \frac{f_\pi}{k} d\mu=\frac{\nu(\Omega)}{k} , $$ 
for every $ \epsilon>0$ there exists  $k$ s.t. $\int_{\{ f_\pi\geq k\}} f_\pi d\mu =\nu(f_\pi\geq k)<\epsilon$.
\end{proof}

One then has to prove that $f_\pi$ converges to $f=\frac{d\nu}{d\mu}$ when the partition $\pi$ becomes finer and finer, in some sense. One way of doing it is to suppose that $(\Omega,d)$ is a compact metric space and $\cF$ is the sigma-algebra of Borel sets, and to define the size of $\pi$ as $|\pi|:=\max_j \textup{diam}(A_j)$, where the diameter of $A \subseteq \Omega$ is defined as $\textup{diam}(A):=\sup_{x,y\in A} d(x,y)$. Notice that, since $\Omega$ is compact, it admits finite partitions $\pi_n$ s.t. $|\pi_n|\to 0$, and given any $f\in L^1(\mu)$ there exist   continuous $c_n$ s.t. $c^n \to f$ in $L^1(\mu)$. 
Since when $c$ is uniformly continuous it follows that $c_\pi\to c$ uniformly as $|\pi|\to 0$, applying the inequality $||h_\pi||_{L^1(\mu)}\leq ||h||_{L^1(\mu)}$ (valid for all $h\in L^1(\mu)$)  to $h=f- c^n$ we get from the triangle inequality
 $$ || f- f_\pi||_{L^1(\mu)}\leq 2|| f- c^n||_{L^1(\mu)} + || c^n - c^n_\pi||_{L^1(\mu)} , $$
which shows  that $f_\pi\to f$ in $L^1(\mu)$ as $|\pi|\to 0$ for any $f\in L^1(\mu)$. One would like however to work in abstract measure spaces; in this case, since there is no notion of size for $\pi$, it is not a priori clear what to do to replace the condition $|\pi|\to 0$. 
If  $\cF$ is separable, i.e. there exists a sequence of  sets $(B_j)_{j\in \bN}$ s.t. $\cF=\sigma((B_j)_{j\in \bN})$, we can consider  the unique partition\footnote{The elements of $\pi_n$ are the atoms of $\sigma((B_j)_{j=0}^n)$, and are the sets of the form $\cap_{j=0}^n C_j$ where $C_j\in \{B_j, \Omega \setminus B_j\}$.} 
$\pi_n$ of $\Omega$ s.t. $\sigma(\pi_n)=\sigma((B_j)_{j=0}^n)$, and then look at the limit of $f_{\pi_n}$ as $n\to \infty$.
This is where the martingale convergence theorem comes in, and ensures that $f_{\tilde{\pi}_n}$ is converging $\mu$ a.e. as $n\to \infty$ for any  \emph{refining} sequence of partitions $(\tilde{\pi}_n)_n$, and in particular for $\tilde{\pi}_n=\pi_n$.
As discussed\footnote{One has to specialize the result stated in Remark \ref{rem: convergence for general nu} to the present case, in which $\cF=\sigma(\cup_n \pi_n)$.} in Remark \ref{rem: convergence for general nu}, given any  $\nu,\mu$,    the $f_{\pi_n}(\mu)$ satisfy
\begin{align}
\label{eq:f_n conv f}
\textstyle
  f_{\pi_n}(\mu) \to \frac{d\nu^a}{d\mu} \quad  \mu \text{ a.e.}.
\end{align} 
 If $\nu \ll \mu$, this allows to prove Theorem \ref{RN} by simply  checking that the limit ($\mu$ a.e. and, by Remark \ref{rem:UI}, in $L^1(\mu)$) $f$ of the  $f_{\pi_n}(\mu)$  is such that $\nu =f\cdot \mu$; this is the core of the probabilistic proof of Theorem \ref{RN}. 
When $\cF$ is not separable however, there is no known equivalent to the interesting result  $f_{\pi_n}(\mu)\to \frac{d\nu^a}{d\mu}$ $\mu$ a.e., and  the traditional probabilistic proof of Theorem \ref{RN} becomes less elementary, as one needs to introduce the concept of $L^1(\mu)$-limit along the net $\cP$ of all partitions of $\Omega$ (ordered by refinement), and show that $(f_{\pi}(\mu))_{\pi\in \cP}$ converges to $\frac{d\nu}{d\mu}$ in $L^1(\mu)$ if $\nu \ll \mu$ (but not otherwise). 
  One reference which explains particularly well how to pass from sequences to nets is \cite[Lemma V-1-1, Proposition V-1-2]{Nev75}. The core of our proof is to avoid using nets altogether.

\section{Proof of Main Theorems}
\label{se:proof}

While we will not assume any knowledge of probability\footnote{In probability theory one considers only the case where $\mu$ has mass $1$; however, the notion of conditional expectation clearly works also for any positive finite $\mu$; in fact, one can even consider the sigma finite case, as done in \cite[Chapter 5.1.2]{Str10}.} theory, to write the proof we find it convenient to use its language, by borrowing the following concept.
Given $f,g\in L^1(\mu)$ and a $\sigma$-algebra $\cG\subseteq \cF$,    denote with $\mu_{|\cG}$ the restriction of $\mu$ to $\cG$. We  call $g$  \emph{the conditional expectation of $f$ given $\cG$} if $g$ is $\cG$-measurable and 
\begin{align}
\label{CondExpDef}
\int_A f d\mu=\int_A g d\mu \quad \text{ for all $A\in \cG$. }
\end{align} 
 Notice that such $g\in L^1(\Omega,\cG,\mu)$, 
	if\footnote{It immediately follows from Theorem  \ref{RN} that $\bE^\mu(f|\cG)$ always exists (indeed it equals $\frac{d\nu_{|\cG}}{d\mu_{|\cG}}$, where $\nu:=f\cdot \mu$), but of course we will not need to use this fact to prove Theorem  \ref{RN}.} 
it exists, is a.e. unique\footnote{If $\tilde{g} \in L^1(\Omega,\cG,\mu)$ satisfies \eqref{CondExpDef} then taking first $A=\{ \tilde{g}>g \}$ and then $A=\{ \tilde{g} <g \}$ in \eqref{CondExpDef} shows that $\mu(\{ \tilde{g} \neq g \})=0$.}; we denote it with $\bE^\mu(f|\cG)$. 
Notice also that  \eqref{CondExpDef}  holds for all $A\in \cG$ if it  holds for all $A$ in a family $\cH \subseteq \cG $ which contains $\Omega$, is closed under pairwise intersections and\footnote{We denote with $\sigma(\cH)$ the smallest $\sigma$-algebra containing $\cH$.}  $\sigma(\cH)=\cG$: indeed in this case $f \cdot \mu$ and $g \cdot \mu$ coincide on $\cH$ and so on $\cG$. Moreover,  $\exists \bE^\mu(f|\cG)=:g$ iff  $g\in L^1(\Omega,\cG,\mu)$ and 
\begin{align}
\label{CondExph}
\int hf d\mu=\int h g d\mu \qquad \text{ for all } h \in L^\infty(\Omega,\cG,\mu) :
\end{align} 
indeed, \eqref{CondExpDef} states that \eqref{CondExph} holds when $h$ is the indicator of a set in $\cG$, and this implies that it also holds  when $h$ is a linear combinations of such indicators, and so by dominated convergence it holds for every $\mu$ a.e. bounded $\cG$-measurable $h$.
Finally, notice that if $g^m=\bE^{\mu}[f^m|\cG]$ for every $m\in \bN$,  $f^m\to f$ and  $g^m\to g$ in $L^1(\mu)$ then $g=\bE^{\mu}[f|\cG]$.

\begin{lemma}
\label{Jensen}
Given $f\in L^1(\mu)$ and a $\sigma$-algebra $\cG\subseteq \cF$,  assume that $\exists \bE^{\mu}(f|\cG)=:g$ and $f\geq 0$. Then $g\geq 0 $, 
\begin{align}
\label{Jensenexp}
\int e^{-f} d\mu \geq \int e^{-g} d\mu , 
\end{align} 
and \eqref{Jensenexp} holds with equality  iff $f=g$.
\end{lemma} 

\begin{proof}
Since $\int_{\{ g<0\}} g d\mu=\int_{\{ g<0\}} f d\mu\geq 0$ it follows that $g\geq 0$. 
Since $\phi(t):=\exp(-t)$ is strictly convex, we have\footnote{Indeed $\phi^{'}$ is strictly increasing and so 
$\phi(t)- \phi(s) = \int_s^t  \phi'(u) du \geq \int_s^t \phi'(s)du= \phi'(s)  (t-s)$, with equality iff $s=t$.}
\begin{align}
\label{convexineq}
\phi(t)- \phi(s) - \phi'(s) (t-s) \geq 0,
\end{align} 
and \eqref{convexineq} holds with equality iff $t=s$.
Taking $s:=g(x),t:=f(x)$ in \eqref{convexineq} and then  integrating\footnote{Since $0\leq f,g \in  L^1(\mu)$, each of the four terms is $\mu$-integrable.} in $d\mu$, it follows from \eqref{CondExph} that \eqref{Jensenexp} holds (since $h:=\phi' \circ g$ is $\cG$-measurable and bounded). Moreover, since a positive function integrates to $0$ iff it equals $0$ a.e., \eqref{Jensenexp} holds with equality  iff $f=g$.
\end{proof}

The traditional probabilistic  proof of Theorem \ref{RN} uses  martingale convergence theorem, or more precisely the following corollary (see \cite[Chapter 14.1]{Wi91}). 
\begin{theorem}
\label{thm:mart conv}
Assume that $(\cF_n)_{n\in \bN}$ is an increasing\footnote{I.e. $\cF_n\subseteq \cF_{n+1}\subseteq \cF$ for all $n\in \bN$.} sequence of sub sigma algebras on $(\Omega, \cF)$, and for all $n\in \bN$,  $f_n\in L^1(\mu)$ is $\cF_n$-measurable and satisfy $\bE^{\mu}[f_{n+1}|\cF_n]=f_n$. If $(f_n)_n$ is uniformly integrable then it converges $\mu$ a.e. and in $L^1(\mu)$ to a $f\in L^1(\mu)$ which satisfies $\bE^{\mu}[f|\cF_n]=f_n$ for all $n\in \bN$.
\end{theorem} 
While relying on Theorem \ref{thm:mart conv} simplifies the proof of Theorem \ref{RN} a little,  this makes the proof non-elementary, and in particular not accessible to analysts. Thus, to get convergence we will rely instead  on the following generalization of the notion of subsequence. Given elements $f_i$ of a vector space, we say that $g$ is a convex combination of $(f_i)_{i\in I}$, and we write $g \in \co((f_{i})_{i\in I})$, if there exists a finite set  $J \subseteq I$  and  for each $i\in J$ an $a_i\geq 0$  such that $\sum_{i\in J} a_i=1$ and $g=\sum_{i\in J} a_i f_{i}$. 
We say that $(g_n)_{n\in \bN}$ is a \emph{forward convex combination} of $(f_n)_{n\in \bN}$, and we write $(g_n)_n \in \fcc((f_{n})_{n})$, if $g_n\in \co((f_k)_{k\geq n})$ for all $n\in \bN$.
Forward convex combinations are important because every subsequence is a forward convex combination, and (similarly to subsequences) satisfy the following two important properties. One, a forward convex combination of a forward convex combination of $(f_n)_{n\in \bN}$ is itself  a forward convex combination of $(f_n)_{n\in \bN}$, i.e. $(h_n)_n \in \fcc((g_{n})_{n})$  and $(g_n)_n \in \fcc((f_{n})_{n})$  imply $(h_n)_n \in \fcc((f_{n})_{n})$. Two, under some boundedness assumptions one can often pass to \emph{converging} forward convex combinations (even when there are no converging subsequences). In this regard, we will make use of the simplest possible result, given in the upcoming Lemma \ref{L2fcc}, a short and completely elementary proof\footnote{To be precise, the given references provide us with $L^2$ convergence. To get also a.e. convergence, we can simply pass to a subsequence.} of which, based on the Hilbert space structure,  can be found in \cite{BeScVe12}; alternatively, notice that Lemma \ref{L2fcc} is a special case of Mazur's lemma, but the proof of the latter relies on the use of the weak topology, and is thus less elementary.
\begin{lemma}
\label{L2fcc}
If $(f_n)_{n\in \bN} $ is bounded in $L^2(\mu)$ then there exists $(g_n)_n \in \fcc((f_{n})_{n})$ such that  $g_n$ converges $\mu$ a.e. and in $L^2(\mu)$.
\end{lemma}

\begin{proof}[Proof of Theorem \ref{RN}.] The  proof of uniqueness of $\nu_a,\nu_s,f$ is standard, see e.g.  \cite[Theorem 6.10]{Ru06}. We will  prove below the existence of $f=\frac{d\nu}{d\mu}$ with values in $[0,1]$ under the  assumption that\footnote{We recall that by definition  two  measures $\alpha,\beta$ satisfy $\alpha\leq \beta$ when $\alpha(A)\leq \beta(A)$ for all $A\in \cF$.} $\nu\leq \mu$; we now derive from this the existence of the decomposition $\nu=\nu^a+\nu^s$ and  of  $f=\frac{d\nu^a}{d\mu}$ for general $\mu,\nu$, by following   \cite[Theorem  7.2.12]{Str98Co}. Since  $\mu+\nu\geq \nu$ there exists $h:\Omega\to [0,1]$ such that  $\nu=h\cdot (\mu + \nu)$, and so 
\begin{equation}
\label{eq:1-h}
\int_\Omega \phi(1 - h) d\nu = \int_\Omega \phi h d\mu  \quad \text{ for all $\phi\geq 0$.}
\end{equation}
Define the continuous bijection $\psi:[0,1]\to [0,\infty]$ by setting $\psi(x):=x/(1-x)$ for $x\in [0,1)$, and $\psi(1):=\infty$.
Set  $g:= \psi\circ h$. Given an arbitrary $\Gamma \in \cF$, apply \eqref{eq:1-h}  with\footnote{Notice that on $\{g<\infty\}=\{h<1\}$ the inverse of $1-h>0$ is well defined.} $\phi=\phi_\Gamma:=  1_{\{g<\infty\} \cap \Gamma} (1-h)^{-1}$ to get
$$ \nu(\Sigma \cap \Gamma) = \int_\Omega \phi_\Gamma(1 - h) d\nu = \int_\Gamma 1_{\{g<\infty\}}gd\mu  $$
i.e. $\nu(\{g<\infty\} \cap \cdot)=g 1_{\{g<\infty\}}\cdot \mu$. Now   \eqref{eq:1-h}  with $\phi := 1_{\{g=\infty\}}$ gives 
$$ \mu(\{g=\infty\}) = \int_{\{h=1\}}hd\mu = \int_{\{h = 1\}}(1-h)d\nu = 0,$$
so $\nu=\nu^s+\nu^a$ with
\begin{align}
\label{eq:nu^a and nu^s and g}
\nu^a:=\nu(\{g<\infty\} \cap \cdot)=g 1_{\{g<\infty\}}\cdot \mu , \quad \nu^s:=\nu((\Omega \setminus \{g<\infty\}) \cap \cdot).
\end{align} 
Let us now  prove the existence of a $[0,1]$-valued $f=\frac{d\nu}{d\mu}$  when $\nu\leq \mu$. We denote with $\cP$ the family of all finite\footnote{One could equivalently work with the family of countable measurable partitions, as the infinite sums which one encounters in the proof trivially converge.} measurable partitions of $\Omega$, and for all $i\in \cP$ let $\cF_i:=\sigma(i)$. Given $i,j\in \cP$, we say that $i\leq j$ if $j$ refines $i$, i.e. if $i \subseteq j$. Notice that given $i,j\in \cP$ there exists\footnote{The  elements of $k$ are the intersections of an element of $i$ with one of $j$.} their supremum $i \vee j$, which is the smallest $k\in \cP$ such that $i\leq k, j\leq k$. 
Given  $A\in \cF$ and $i\in \cP$, we denote with $i\vee A\in \cP$ the element $i\vee j$ where $j:=\{ A, \Omega \setminus A\}\in \cP$. 
For $i=\{A_k\}_{k=1,\ldots, m}\in \cP$, define $f_i$ by
setting 
\begin{align}
\label{eq: fi def using nu}
 f_i:= f_i(\mu):=\sum_{k} 1_{A_k} \frac{\nu(A_k)}{\mu(A_k)} ,
\end{align} 
where the sum is taken over the $k\in \{1,\ldots, m\}$'s such that $\mu(A_k)>0$.
 Notice that $f_i$ has values in $[0,1]$, is $\cF_i$-measurable and satisfies
\begin{align}
\label{nu=intfi}
 \textstyle{\nu(A)=\int_A f_i d\mu }
\end{align} 
 for all $A\in i$, and thus for all $A\in \cF_i$. 
 The idea of the classic probabilistic proof is to show that  $(f_{i})_i$ converges   in $L^1(\mu)$ to some $f$, so that 
\begin{align}
\label{nufig}
 \textstyle{\nu(A)=\int_A f_i d\mu=\int_A f d\mu  } 
\end{align} 
holds for all  $A\in \cF=\cup_i \cF_{i}$; however, since we  only want to deal with convergence along a sequence $(i_n)_n$, and in general\footnote{In particular whenever $\cF$ is not separable.} the inclusion $ \cF \subset \sigma(\cup_n \cF_{i_n})$ is strict, we have to be more careful.

Let $i,j\in \cP$, $j\geq i$, then  \eqref{nu=intfi} implies that $f_i=\bE^\mu[f_j|\cF_i]$, so by   Lemma \ref{Jensen} the   family $a_i:=\int e^{-f_i} d\mu, i\in \cP$ is increasing. Choose  $(i_n)_{n\in \bN} \subseteq  \cP$ such that 
\begin{align}
\label{ChoiceofSub}
\textstyle{a_{i_n}\uparrow \sup_{i\in \cP} a_i=:a.}
\end{align} 
Since $\cP$ is directed\footnote{A partially ordered set $(I,\leq)$ is said to be (upward-)directed if for any $i,j\in I$ there exists $k\in I$ s.t. $i\leq k$ and $j\leq k$.},  Lemma \ref{Jensen} implies that we can assume w.l.o.g\footnote{Otherwise replace $(i_n)_n$ with $(j_n)_n$ defined by induction by setting $j_0:=i_0$ and $j_{n+1}:=j_n\vee i_{n+1}$. Then  Lemma \ref{Jensen} implies that $a_{i_n} \leq a_{j_n} \leq a_{j_{n+1}}$, and so $a_{j_n}\uparrow a$.} that $(i_n)_n$ is increasing, and so such is $(\cF_{i_n})_{n}$.
 Using Lemma \ref{L2fcc} we now choose $(g_n)_n \in \fcc((f_{i_n})_{n})$ which converges $\mu $ a.e. to some $f$.
Since $f_{i_n}$ takes values in $[0,1]$, so do $g_{n}$ and $e^{-g_{n}}$; it follows that $g_{n} \to f$ and $e^{-g_{n}} \to e^{-f}$ in $L^1(\mu)$, and $f$ takes values in $[0,1]$. 
Now fix an arbitrary $A\in \cF$ and define $\tilde{f}_i:=f_{i \vee A}, \tilde{\cF}_i:=\cF_{i \vee A}$. 
Proceeding as before we can choose a $(\tilde{g}_n)_n\in \fcc((\tilde{f}_{i_n})_{n})$ such that $\tilde{g}_n$ converges $\mu$ a.e. and in $L^1(\mu)$ to some $\tilde{f}$, and  $e^{-\tilde{g}_{n}} \to e^{-\tilde{f}}$  in $L^1(\mu)$.
 Thus 
\begin{align}
\label{LimInt}
\int e^{-g_{n}} d \mu \to \int e^{-f} d \mu \quad \text{ and } \quad \int e^{-\tilde{g}_{n}} d \mu \to \int e^{-\tilde{f}} d \mu  .
\end{align} 
 Notice that $f_i=\bE^\mu[f_j|\cF_i]$ for $i\leq j$ implies that 
 \begin{align}
\label{ChainCondExp}
 f_{i_n}=\bE^\mu[\tilde{f}_{i_n}|\cF_{i_n}], \quad  \tilde{f}_{i_n}=\bE^\mu[\tilde{g}_{m}|\tilde{\cF}_{i_n}], \quad  f_{i_n}=\bE^\mu[g_{m}|\cF_{i_n}]  \quad \text{ for all $ m\geq n$} .
\end{align} 
 It follows from \eqref{ChainCondExp} and Lemma \ref{Jensen} that 
 \begin{align}
\label{chainineq}
\int e^{-f_{i_n}} d \mu \leq \int e^{-\tilde{f}_{i_n}} d \mu \leq \int e^{-\tilde{g}_{n}} d \mu , \qquad
\int e^{-f_{i_n}} d \mu \leq \int e^{-g_{n}} d \mu .
\end{align} 
Since $a_i\leq a$ for every $i$, the convexity of $e^{-\cdot}$ implies that 
$$ \int e^{-g_{n}} d \mu \leq a, \qquad \int e^{-\tilde{g}_{n}} d \mu \leq a ,$$
which together with $a_{i_n}\to a$, \eqref{LimInt} and \eqref{chainineq} give that 
\begin{align}
\label{EqualInt}
\int e^{-\tilde{f}} d \mu  =a= \int e^{-f} d \mu  .
\end{align} 
Taking $m\to \infty$ in \eqref{ChainCondExp} we get that 
\begin{align}
\label{ChainCondExp2}
f_{i_n}=\bE^\mu[\tilde{f}_{i_n}|\cF_{i_n}], \quad  \tilde{f}_{i_n}=\bE^\mu[\tilde{f}|\tilde{\cF}_{i_n}], \quad  f_{i_n}=\bE^\mu[f|\cF_{i_n}] ,
\end{align}  
and so the following equalities hold for every $B\in \cF_{i_n}\subseteq \tilde{\cF}_{i_n}$
\begin{align}
\label{fvvfif}
\int_B \tilde{f} d\mu= \int_B  \tilde{f}_{i_n} d\mu= \int_B  f_{i_n } d\mu=\int_B f d\mu .
\end{align} 
From \eqref{ChainCondExp2}  and $A\in \tilde{\cF}_{i_n}$ it follows that $\int_A \tilde{f} d\mu=\int_A  \tilde{f}_{i_n} d\mu$; thus using also   \eqref{nu=intfi} with $i:=i_n\vee A$ we get that
\begin{align}
\label{fvvAnu}
\int_A \tilde{f} d\mu=\nu(A) .
\end{align} 
Notice that \eqref{fvvfif} shows that $\int_B \tilde{f} d\mu=\int_B f d\mu $
holds for every $B$ in the algebra $\cH:=\cup_n \cF_{i_n}$; but then it holds for every $B\in \cG:=\sigma(\cup_n \cF_{i_n})$, and since $f\in L^1(\Omega, \cG,\mu)$ this means that $f=\bE^\mu[\tilde{f}|\cG]$.
 It then follows from \eqref{EqualInt} and Lemma \ref{Jensen}  that  $ \tilde{f}= f$, thus \eqref{fvvAnu} implies that $\int_A f d\mu=\nu(A) $. 
\end{proof}

\begin{remark}
One can modify the above proof of Theorem \ref{RN} to show, without using Theorem \ref{thm:mart conv},  that any uniformly integrable martingale $(M_i, \cF_i)_{i\in I}$ (indexed by an
	 arbitrary directed set $I$) is closed. 
	 This is done showing that there exists an increasing sequence of indices $i_n \in I, n\in \bN$ for which  $(M_{i_n})_{n\in \bN}$ admits a forward convex combination converging in $L^1$ to some $M_{\infty}$ which closes\footnote{Meaning that $M_i=\bE[M_\infty|\cF_i]$  for all $i\in I$.}  $(M_i, \cF_i)_{i\in I}$, by applying the following changes to the proof of Theorem \ref{RN}: replace Lemma \ref{L2fcc} with  \cite[Lemma 2.1]{BeScVe12}, and the function $\exp(- \cdot)$ in Lemma \ref{Jensen} with $\phi(t):=\int_0^t \tan^{-1}(x) dx$. Then the proof goes through unchanged, since  $\phi$ is strictly convex, even, and satisfies $\phi\geq 0$  and $|\phi'|\leq 1$, and   
$(\phi(M_i))_{i\in I}$ is uniformly integrable (and so such is $(g_n)_{n\in \bN}$  for any $g_n \in \fcc (\phi(M_{i_n}))_{n\in \bN}$) since $0\leq \phi(x)\leq |x|$.
\end{remark} 

\begin{proof}[Proof of Theorem \ref{thm: conv seq}]
By Remark \ref{rem:UI} the  sequence $(f_{i_n}(\gamma))_n$ is uniformly integrable, and by \eqref{nu=intfi} it satisfies
$\bE^{\mu}[f_{i_{n+1}}(\gamma)|\cF_n]=f_{i_n}(\gamma)$ for all $n$. 
  Theorem \ref{thm:mart conv} gives that $(f_{i_n}(\gamma))_n$ 
 converges $\gamma$ a.e.  to some $h$, and our proof of Theorem \ref{RN} shows that $(f_{i_n}(\gamma))_n$  admits a forward convex combination which converges in $L^1(\gamma)$ to $\frac{d\nu}{d\gamma}$. It follows that  $f(\gamma):=\frac{d\nu}{d\gamma}=h$ is the  limit $\gamma$ a.e. and in $L^1(\gamma)$ of $(f_{i_n}(\gamma))_n$. 
Let $\psi$ be the function defined  after \eqref{eq:1-h}, and  recall the fact, stated in  \eqref{eq:nu^a and nu^s and g}, that   $g(\gamma):= \psi\circ f(\gamma)$
satisfies  $g(\gamma)1_{\{g(\gamma)<\infty\}}=\frac{d\nu^a}{d\mu}$ and $\mu(\{g(\gamma)=\infty\})=0$. Analogously, using also \eqref{eq:RN density}, 
 $f_{i_n}(\gamma)=\frac{d\nu_{|i_n}}{d\gamma_{|i_n}}$ and  $g_{i_n}(\gamma):= \psi \circ f_{i_n}(\gamma) $ satisfy 
$$g_{i_n}(\gamma) 1_{\{g_{i_n}(\gamma)<\infty\}}=\frac{d(\nu_{|i_n})^a}{d\mu_{|i_n}}=f_{i_n}(\mu) .$$ 
Since $\psi$ is continuous, $g_{i_n}(\gamma)\to g(\gamma)$ $\gamma$ a.e. and thus, since $g(\gamma)<\infty$ $\mu$ a.e.,  
$$f_{i_n}(\mu)=g_{i_n}(\gamma) 1_{\{g_{i_n}(\gamma)<\infty\}} \to g(\gamma) 1_{\{g(\gamma)<\infty\}} =\frac{d\nu^a}{d\mu} \text{ $\mu$ a.e.. }$$ 
\end{proof}

\section{How to properly choose the sequence of partitions}
\label{se:analogies}

The following discussion  is meant to make more intuitive the choice of the sequence of partitions $(i_n)_n$ made in \eqref{ChoiceofSub} during our proof of Theorem \ref{RN}, and to highlight interesting similarities between our proof and the most popular proof. 
The $\mu$-essential supremum of a family $\cH$ of measurable functions from $\Omega$ to $[0,1]$ is simply the supremum of $\cH$ in the order of $L^1(\mu)$, a the space of equivalence classes.  This is built as the pointwise supremum of any sequence $(g_n)_n$  which  asymptotically maximizes the functional $g\mapsto \int_{\Omega} g d\mu$; moreover  the sequence can  w.l.o.g.  be assumed increasing
if, as is usually the case, $\cH$ is directed. 
Thus, the most popular proof of  Theorem \ref{RN}  builds $f=\frac{d\nu^a}{d\mu}$ as the $L^1(\mu)$ limit of any increasing sequence $g_n \in L(\mu,\nu), n\in \bN$ which  asymptotically maximizes the functional $g\mapsto \int_{\Omega} g d\mu$. Analogously, for $\nu \leq \mu$ our proof of Theorem \ref{RN} builds $f$ as the $L^1(\mu)$ limit of\footnote{To be precise, the way the proof is written shows only the convergence of a forward convex combination of $(f_{i_n})_n$; the fact that the whole sequence converges is due to Theorem \ref{thm:mart conv}.} $f_{i_n}, n\in \bN$, where   $i_n\in \cP, n\in \bN$ is increasing and is chosen such that $(f_{i_n})_n$  asymptotically maximizes the function $i\mapsto \int_{\Omega} e^{-f_i} d\mu$. 

Crucially for our proof, $(f_{i})_{i\in \cP}$  is  `increasing' in the following sense. Let $\cS$ be the family of all $ \sigma\text{-algebras}  $ $\cG \, \subseteq \cF$; we define an order on the set 
\begin{align}
\label{G}
\bG:=\{G:=(g,\cG): \cG \in \cS, g\in L^1(\Omega, \cG, \mu) \} 
\end{align} 
by saying that, given $G_i:=(g_i,\cG_i) \in \bG$,  $G_1\leq G_2$ holds if $\cG_1 \subseteq \cG_2$
 and $g_1=\bE^{\mu}[g_2|\cG_1]$. Then $\leq$ is a (partial) order on $\bG$, which induces the  order of inclusion of sets $\cG_1 \subseteq \cG_2$ on $\cS$, and the convex order $\leq_c$ on the image measures $\mu(g_1^{-1}(\cdot))\leq_c \mu(g_2^{-1}(\cdot))$. 
The above order on $\bG$ is important because saying that a map $H: \cP \ni i \mapsto (h_i,\cF_i) \in \bG$ 
 is increasing is simply a way of stating that  $M:=((h_i,\cF_i))_{i\in \cP} $  is a martingale, and then   $H$ being bounded above means that  $M$  is closed. Thus, $(f_{i})_{i\in \cP}$ is increasing in the following sense:  $f_i$ can be identified with $(f_i,\cF_i)$ where $\cF_i:=\sigma(i)$,  and $\cP \ni  i \mapsto (f_{i}, \cF_{i}) \in \bG$  is increasing. 
Our proof of Theorem \ref{RN} could then be rewritten in the above language, which would probably make the proof somewhat more intuitive, though not as transparent.

Since $i \mapsto (f_{i}, \cF_{i}) $ and  $\bG \ni (g,\cG)\mapsto \int_{\Omega} e^{-g} d\mu \in [0,1]$ are increasing (by Lemma \ref{Jensen}), their composition  $i\mapsto \int_{\Omega} e^{-f_i} d\mu$ is increasing, just like in  most popular proof the function $g\mapsto \int_{\Omega} g d\mu$ is  increasing, and this is why we choose them as functions to maximize. 
 
Just like  $f$ is  the supremum in $L^1(\mu)$ of  $L(\mu,\nu)$ (and of the increasing family $(f_n)_n$, defined in \eqref{eq:f_n esssup}), 
so also $(f,\cF)$ is the supremum in $\bG$ of $(f_{i}, \cF_{i})_{i\in \cP}$.
 When dealing with the order on $\bG$ instead of on $L^1(\mu)$, there are however two related complications. 
 One is that, while $f$ is always the supremum in $L^1(\mu)$ of $\{g_n\}_n$, $(f,\cF)$ does not in general equal the supremum $s$ in $\bG$ of $(f_{i_n}, \cF_{i_n})_n$; indeed,  $s=(f,\tilde{\cF})$, where $\tilde{\cF}:=\sigma(\cup_n \cF_{i_n})$; thus $(f,\cF)=s$ iff $\cF=\tilde{\cF}$, i.e. iff $\cF$ is countably generated and $i_n$ is chosen\footnote{Notice that such a choice is always compatible with the requirement of maximizing $F(g,\cG):=\int_{\Omega} e^{-g} d\mu $, since $F$ is increasing. More precisely, if $(i_n)_n$ and $(j_n)_n$ are increasing and satisfy $\cF=\sigma(\cup_n \cF_{i_n})$ and $F(f_{j_n},\cF_{j_n})\uparrow \sup_{i\in \cP} F(f_i,\cF_i)$,  then $k_n=i_n \vee j_n$ is increasing and satisfies $\cF=\sigma(\cup_n \cF_{k_n})$ and $F(f_{k_n},\cF_{k_n})\uparrow \sup_{i\in \cP} F(f_i,\cF_i)$, since $F(f_{i_n},\cF_{i_n}) \leq F(f_{k_n},\cF_{k_n})$. 
 Thus, there exists an appropriate choice of $(i_n)_n$ (so that $\cF=\tilde{\cF}$) iff $\cF$ is separable.} so that $\cF=\sigma(\cup_n \cF_{i_n})$. 
The other complication is that, while (when $0 \leq \nu \leq \mu$) the functions 
$$L^1(\mu) \in g \mapsto \int g d\mu \in [0,1] \quad \text{  and} \quad \cP \ni i \mapsto (f_i,\cF_i) \in \bG$$ are \emph{strictly} increasing,  $\bG \ni (g,\cG) \mapsto \int_{\Omega} e^{-g} d\mu \in [0,1]$
 is increasing but not strictly; instead, one relies on the weaker property expressed in Lemma \ref{Jensenexp}.

One could try to  make these two subtleties evaporate by endowing not $\bG$ but rather $L^1(\mu)$ with an order, defining $g_1\leq g_2$ if $\bE[g_2|\sigma(g_1)]=g_1$.
However, $\cP \ni i \mapsto f_i \in L^1(\mu)$ and  $\cP \ni i\mapsto \int_{\Omega} e^{-f_i} d\mu \in [0,1]$ are also increasing but not strictly, and it is important in our proof that $A\in \cF_{i \vee A}$, and since in general $A$ does not belong to $\sigma(f_{i \vee A})$, we have to consider the order on $\bG$, not on $L^1(\mu)$.

\medskip

{\bf Acknowledgment.} Oleksii Mostovyi is supported by the National Science Foundation under grant No. DMS-1848339 (2019-2024). Any opinions, findings and conclusions or recommendations expressed in this material are those of the author(s) and do not necessarily reflect the views of the National Science Foundation.

\bibliography{Radon_Nikodym}{}
\bibliographystyle{abbrv}

% a common alternative is 
%\bibliographystyle{alpha}

% the following is here just so that once the paper is written I can include the bibliography in the file if I want to

%  \begin{thebibliography}{10}

% to give an example I write an entry here 

%\bibitem{Bass10}  R.~F. Bass. \newblock The measurability of hitting times.   \newblock {\em Electronic Communications in Probability}, 15:99--105, 2010.

%  \end{thebibliography}

\end{document}